\documentclass{birkjour}
 \newtheorem{thm}{Theorem}[section]
 
 \newtheorem{lem}[thm]{Lemma}
 
 \newtheorem*{theo}{Theorem}
 \theoremstyle{definition}
 \newtheorem{defn}[thm]{Definition}
 \theoremstyle{remark}
 \newtheorem{rem}[thm]{Remark}

 \usepackage[mathscr]{eucal}
 \usepackage[OT1,T1,TS1,T2A]{fontenc}
\usepackage[cp1251]{inputenc}
\usepackage[russian,english]{babel}
 \numberwithin{equation}{section}
\newcommand{\ccomma}{\mathpunct{\raisebox{0.5ex}{,}}}

\DeclareMathOperator{\tr}{trace}

\begin{document}

%
%
%
%
%

\title[The matrix function \(e^{tA+B}\)]
 {The matrix function \(\boldsymbol{e^{tA+B}}\) is representable as
 the Laplace transform of a matrix measure.}
\author[V. Katsnelson]{Victor Katsnelson}
\address{%
Department of Mathematics\\
The Weizmann Institute\\
76100, Rehovot\\
Israel}
\email{victor.katsnelson@weizmann.ac.il; victorkatsnelson@gmail.com}
\date{July 1, 2016}
\begin{abstract}
Given a pair \(A,B\) of matrices of size \(n\times n\), we consider the matrix function \(e^{tA+B}\) of the variable  \(t\in\mathbb{C}\). If the matrix \(A\) is Hermitian, the matrix function   \(e^{tA+B}\) is representable
as the bilateral Laplace transform of a matrix-valued measure \(M(d\lambda)\) compactly supported on the real axis:
\[e^{tA+B}=\int{}e^{t\lambda}\,M(d\lambda).\]
 The values of the measure \(M(d\lambda)\) are matrices of size \(n\times n\), the support of this measure is contained in
 the convex hull of the spectrum of \(A\). If the matrix \(B\) is also Hermitian,
then the values of the measure  \(M(d\lambda)\) are Hermitian matrices. The measure \(M(d\lambda)\) is not necessarily non-negative.
\end{abstract}
\maketitle

\textbf{Notation}\\
\(\mathbb{C}\) is the set of complex numbers.\\
\(\mathbb{R}\) is the set of real numbers.\\
\(\mathbb{R}^{+}\) is the set of non-negative real numbers.\\
\(\mathbb{N}\) is the set of natural numbers.\\
\(\mathfrak{M}_{n}\) is the set of matrices of size \(n\times{}n\) with entries belonging to \(\mathbb{C}\).\\
\(\mathfrak{M}_{n}^{+}\) is the set of matrices of size \(n\times{}n\) with entries belonging to \(\mathbb{R}^+\).\\
\(\mathfrak{H}_{n}\) is the set of Hermitian matrices of size \(n\times{}n\).\\
\(\mathfrak{D}_{n}\) is the set of diagonal matrices of size \(n\times{}n\).\\
\(I_n\) is the identity matrix of size \(n\times{}n\).\\
We provide the set \(\mathfrak{M}_{n}\) with the usual algebraic operations - the matrix addition and the matrix multiplication.\\[0.9ex]

\section{The goal of the present paper.}
\label{sec1}
\noindent

Let \(A\in\mathfrak{H}_n\) and \(B\in\mathfrak{M}_n\). In the present paper
we consider the matrix function \(L(t)=e^{tA+B}\) of the complex variable \(t\).
We show that this function is representable as the bilateral Laplace transform of
some matrix valued measure \(M(d\lambda)\):
\begin{equation}
\label{LR}
e^{tA+B}=\int{}e^{t\lambda}\,M(d\lambda),\quad t\in\mathbb{C},
\end{equation}
the values  of the measure \(M\) belong to the set \(\mathfrak{M}_n\).

Our considerations are based on the functional calculus for the matrix~\(A\).
We relate the following objects with the matrix \(A\):\\[0.5ex]
1.\,The spectrum \(\sigma(A)\) of the matrix \(A\), that is the set \(\{\lambda_1,\,\ldots\,,\lambda_l\}\) of all its eigenvalues taken without
multiplicities, i.e. \(\lambda_p\not=\lambda_q,\,\forall\,p\not=q,1\leq p,q\leq l\).
Since \(A\in\mathfrak{H}_n\),
\(\sigma(A)\subset\mathbb{R}\).
(The number \(l\) is the cardinality of the set \(\sigma(A)\), \(l\leq{}n.\) If
the spectrum \(\sigma(A)\) is simple, then \(l=n\).)
\\
2. The set \(\{E_{\lambda_1},\,\ldots\,,E_{\lambda_l}\}\) of spectral projectors
of the matrix \(A\):
\begin{gather}
\label{SP}
AE_{\lambda_j}=\lambda_jE_{\lambda_j},\quad 1\leq j\leq l,\\
E_{\lambda_1}+\,\cdots\,+E_{\lambda_l}=I_n.
\end{gather}

If \(f(\lambda)\) is a function defined on the spectrum \(\sigma(A)\), then
\begin{equation}
\label{f(A)}
f(A)=\sum\limits_{1\leq j\leq l}f(\lambda_j)E_{\lambda_j}.
\end{equation}
In particular,
\begin{equation}
\label{e(A)}
e^{tA}=\sum\limits_{1\leq j\leq l}e^{t\lambda_j}E_{\lambda_j}.
\end{equation}
If the matrices \(A\) and \(B\) commute, that is if
\begin{equation}
\label{com}
AB=BA,
\end{equation}
then
\begin{equation}
\label{eom}
e^{tA+B}=e^{tA}\cdot e^{B}.
\end{equation}
From \eqref{e(A)} and \eqref{eom} it follows that under the condition \eqref{com} the equality
\begin{equation}
\label{ir}
e^{tA+B}=\sum\limits_{1\leq j\leq l}e^{t\lambda_j}M(\{\lambda_j\})
\end{equation}
holds, where
\begin{equation}
\label{am}
M(\{\lambda_j\})=E_{\lambda_j}e^{B}E_{\lambda_j}.
\end{equation}
 \emph{The equality \eqref{ir} can be interpreted as the representation of the
matrix function \(e^{tA+B}\) in the form of the bilateral Laplace transform \eqref{LR} of a very special
matrix valued measure \(M\). This measure \(M\) is discrete and is supported
 on the spectrum \(\sigma(A)\) of the matrix \(A\). The point \(\{\lambda_j\}\in\sigma(A)\) carries the "atom" \(M(\{\lambda_j\})\)}.

\emph{The goal of the present paper is to obtain the representation of the matrix function
\(e^{tA+B}\) in the form \eqref{LR} without assuming that the matrices \(A\) and \(B\) commute.}

The representation of the form \eqref{LR} was suggested by the following
\begin{theo} \textup{(H.Stahl)}.
Let  matrices \(A\) and \(B\) be given, \(A\in\mathfrak{H}_n\), \(B\in\mathfrak{H}_n\). Let \(\lambda_{\textup{min}}\) and
    \(\lambda_{\textup{max}}\) be the smallest and the largest eigenvalues of
    the matrix~\(A\).
    Then the function \(\tr{}e^{At+B}\) is representable in the form
    \begin{equation}\label{StR}
    \tr{}e^{tA+B}=\!\!\int\limits_{[\lambda_{\textup{min}},\lambda_{\textup{max}}]}\!\!
    e^{t\lambda}\mu(d\lambda),
    \end{equation}
    where \(\mu(d\lambda)\) is a non-negative Borel measure.
\end{theo}
The first arXiv version of Stahl's Theorem appeared in \cite{S1}, the latest arXiv version -
in \cite{S2}, the journal publication - in \cite{S3}.
The proof of Stahl is based on ingenious considerations related to
Riemann surfaces of algebraic functions. In \cite{E1},\cite{E2} a simplified version of Stahl proof is presented. The proof, presented in \cite{E1},\cite{E2}, preserves all the main ideas of Stahl; the simplification
consists in technical details. In the paper \cite{K} a proof of Stahl's Theorem for the special case
\(\textup{rank}\,A=1\) is presented. This proof is based on an elementary argument which
does not require complex analysis.

The main result of the present paper is Theorem \ref{MaT}. Stahl's Theorem does
not follow from our Theorem \ref{MaT}. If \(A\in\mathfrak{H}_n\) and \(B\in\mathfrak{H}_n\), then the measure \(M(d\lambda)\) in \eqref{LR} is \(\mathfrak{H}_n\)-valued but not necessarily is non-negative. An appropriate example is given in Section \ref{NNN}.

\section{The approximant \(\boldsymbol{L_{N}(t)}\).}
\label{LAp}
If the matrices \(A\) and \(B\) do not commute, then the equality \eqref{eom} breaks down.
However the Lie product formula, which is a kind of surrogate for the formula \eqref{eom}, holds regardless of the condition \eqref{com}.\\

\noindent
\textbf{Lie Product Formula.} \emph{Let \(X\in\mathfrak{M}_n\) and \(Y\in\mathfrak{M}_n\).  Then}
\begin{equation}
\label{LPrFo}
e^{X+Y}=\lim\limits_{N\to\infty}\big(e^{X/N}e^{Y/N}\big)^{N}.
\end{equation}
 Versions of proof of the Lie Product Formula\footnote{
 It should be mention that the Lie Product Formula can be extended to certain unbounded linear operators \(X\) and \(Y\) acting in a Hilbert space. First such extension was done
by Trotter, \cite{Tr}. A version of the Trotter-Lie product formula was obtained by
T.Kato, \cite{Ka1}.
We refer also to the book of B.\,Simon, \cite{Si}. See Theorems 1.1 and 1.2 there.
In \cite{Si}, the Lie-Trotter formula is used for the path integral formulation of quantum mechanics.
 } can be found in \cite{H}, \cite[Section 6.5]{HJ}, \cite[Theorem 2.10]{Ha}, \cite[Section 2.12, Corollary 2.12.5]{Va}.\\[2.0ex]

\begin{lem}
Let \(A\in\mathfrak{M}_n,\,B\in\mathfrak{M}_n\) and \(t\in\mathbb{C}\). Then
\begin{equation}
\label{LPF}
e^{tA+B}=\lim\limits_{N\to\infty}\Big(e^{tA/N}e^{B/N}\Big)^{N}.
\end{equation}
\end{lem}
\begin{proof}
The formula \eqref{LPF} is a special case of the formula \eqref{LPrFo}
corresponding to the choice \(X=tA,\,Y=B\).
\end{proof}
\begin{defn}
The   expression
\begin{equation}
\label{LA}
L_N(t)=\Big(e^{tA/N}e^{B/N}\Big)^{N}
\end{equation}
which appears on the right hand side of \eqref{LA} is said to be the \emph{N-approximant for the
matrix function \(L(t)=e^{tA+B}\). }
\end{defn}

Assuming that \(A\in\mathfrak{H}_n\), we express the matrix function \(e^{tA/N}\)
in terms of the spectrum \(\sigma(A)\) of the matrix \(A\) and its spectral projectors:
\begin{equation}
\label{es}
e^{tA/N}=\sum\limits_{1\leq j\leq l}e^{t\frac{\lambda_j}{N}}E_{\lambda_j}.
\end{equation}
Substituting \eqref{es} into \eqref{LA}, we represent the approximant \(L_{N}(t)\) as a multiple sum which contains \(l^N\) summands:
\begin{equation}
\label{EEL}
L_{N}(t)=\sum\limits_{k_1,\ldots,k_N}\textup{exp}\,
\Big\{t\,\tfrac{\lambda_{k_1}+\,\ldots\,+\lambda_{k_N}}{N}\Big\}M_{k_1,\,\ldots,\,k_N}.
\end{equation}
In \eqref{EEL}, the summation is extended over all integers \(k_1,\ldots,k_N\) such that
\(1\leq~k_p\leq~l\), \(p=1,2,\ldots,N.\) The matrix \(M_{k_1,\,\ldots,\,k_N}\)
is the product
\begin{equation}
\label{Pr}
M_{k_1,\,\ldots,\,k_N}=E_{\lambda_{k_1}}e^{B/N}E_{\lambda_{k_2}}e^{B/N}\,\ldots\,
E_{\lambda_{k_N}}e^{B/N}.
\end{equation}
Let us consider the numbers \(\tfrac{\lambda_{k_1}+\,\ldots\,+\lambda_{k_N}}{N}\)
which appear in the exponents of the exponentials in \eqref{EEL}.
\begin{lem}
Given integers \(k_1,\ldots,k_N\) satisfying the conditions
\(1\leq k_p\leq l\), \(p=1,2,\ldots,N\), then
\begin{equation}
\label{rlc}
\frac{\lambda_{k_1}+\,\ldots\,+\lambda_{k_N}}{N}=\frac{n_1}{N}\lambda_1+
\frac{n_2}{N}\lambda_2+\,\cdots\,+\frac{n_l}{N}\lambda_l,
\end{equation}
where
\begin{equation}
\label{dnj}
n_j(k_1,\ldots,k_N)=\#\{p:\,1\leq p\leq N,\,k_p=j\},\quad 1\leq j\leq l.
\end{equation}
The numbers
\begin{equation}
\xi_j=\frac{n_j}{N},\,j=1,2,\,\ldots\,l,
\end{equation}
where the \(n_j\) are defined by \eqref{dnj},
 satisfy the conditions
 \begin{equation}
 \label{CC}
\xi_j\geq0,\,j=1,2,\,\ldots\,l,\quad\sum\limits_{1\leq{}j\leq{}l}\xi_j=1.
\end{equation}
\end{lem}
\begin{proof}The lemma is evident.
\end{proof}
The linear combination \(\xi_1\lambda_1+\xi_2\lambda_2+\,\cdots\,\xi_l\lambda_l\)
which appears on the right hand side of \eqref{rlc} is a \emph{convex} linear combination
of numbers \(\lambda_1,\lambda_2,\,\ldots\,,\lambda_l\). However this linear combination
is a very special convex linear combination. Its coefficients \(\xi_1,\,\ldots\,\xi_l\)
are numbers of the form \(\xi_j=\frac{n_j}{N}\), where \(n_j\) are non-negative integers.
\begin{defn}\label{dNch}
 Let \(\lambda_1,\,\ldots\,,\lambda_l\)  be real numbers, \(N\) be a positive
integer. \emph{The \(N\)-convex hull of the set} \(\{\lambda_1,\,\ldots\,,\lambda_l\}\)
 is the set \(\xi_1\lambda_1+\xi_2\lambda_2+\,\cdots\,\xi_l\lambda_l\) of all convex linear combinations with coefficients of the form \(\xi_j=\frac{n_j}{N}\),
 where \(n_j\) are non-negative integers. (Since the considered linear combinations are
 convex, the equality \(n_1+n_2+\cdots+n_l=N\) must hold.)
 \end{defn}

 In what follows, the numbers \(\lambda_1,\,\ldots\,\lambda_l\) form the spectrum
 \(\sigma(A)\) of the matrix \(A\).
 The \(N\)-convex hull of
 the spectrum \(\sigma(A)\) is denoted by \(ch_N(\sigma(A))\).
 The convex hull of the spectrum \(\sigma(A)\) is denoted by \(ch(\sigma(A))\).
 \begin{rem}
 \label{con}
 It is clear that the convex hull \(ch(\sigma(A))\) is the closed interval
 \([\lambda_{\textup{min}},\lambda_{\textup{max}}]\), where
 \(\lambda_{\textup{min}}=\min\limits_{1\leq{}j\leq{}l}\lambda_j\),
 \(\lambda_{\textup{max}}=\max\limits_{1\leq{}j\leq{}l}\lambda_j\).
 It is also clear, that
 \begin{equation}
 \label{inc}
 ch_N(\sigma(A))\subset {}ch(\sigma(A)),\quad \forall\,N.
 \end{equation}
 The union \(\bigcup\limits_{N}ch_N(\sigma(A))\)
 of the sets \( ch_N(\sigma(A)\) is dense in the set
\(ch(\sigma(A))\).
 \end{rem}

 The numbers \(\tfrac{\lambda_{k_1}+\,\ldots\,+\lambda_{k_N}}{N}\) which appear in
 the exponents of the exponentials in \eqref{EEL} belong
 to the set \(ch_N(\sigma(A))\). Collecting similar terms, we rewrite \eqref{EEL} in the
 form
\begin{equation}
\label{EEa}
L_N(t)=\sum\limits_{\lambda\in{}ch_N(\sigma(A))}e^{t\lambda}\,M_N(\{\lambda\}),
\end{equation}
where
\begin{equation}
\label{EEb}
M_N(\{\lambda\})=\sum\limits_{k_1,\ldots,k_N}M_{k_1,\ldots,k_N},
\end{equation}
the matrices \(M_{k_1,\ldots,k_N}\) are defined by \eqref{Pr}.
For each \(\lambda\in{}ch_N(\sigma(A))\), the sum in \eqref{EEb}
is extended over all those  \(k_1,\ldots,k_N\) for which
\(\tfrac{\lambda_{k_1}+\,\ldots\,+\lambda_{k_N}}{N}=\lambda\).

We interpret the equality \eqref{EEa} as the representation   of  the approximant \(L_N(t)\) in the form of the
bilateral Laplace transform  of a matrix valued measure \(M_N(d\lambda)\):
\begin{equation}
\label{LRN}
L_N(t)=\int\limits_{\lambda\in{}ch_N(\sigma(A))} e^{t\lambda}\,M_N(d\lambda).
\end{equation}
The measure \(M_N(d\lambda)\) is discrete and is supported on the finite set
\(ch_N(\sigma(A))\).
The point \(\{\lambda\}\in{}ch_N(\sigma(A))\) carries the "atom" \(M_N(\{\lambda\})\).

According to \eqref{LPF}
\begin{equation}
\label{LRI}
e^{tA+B}=\lim\limits_{N\to\infty}\int e^{t\lambda}\,M_N(d\lambda),\quad \forall t\in\mathbb{C}.
\end{equation}
\section{The norm in the set \(\mathfrak{M}_n\).}
\label{sec3}
 We have to prove that the sequence \(\big\{M_{N}(d\lambda)\big\}_{1\leq{}N<\infty}\) of matrix measures is weakly convergent. To prove this, we have to bound the total variations of these measures from above.

To express such bound, we have to provide the set \(\mathfrak{M}_n\) with some norm.
We provide the set  \(\mathfrak{M}_{n}\) with the usual \emph{operator norm}. Let
\(S=(s_{pq})_1^n\in\mathfrak{M}_{n}\). The norm \(\|S\|\) is defined as follows:
\begin{equation}
\label{DON}
\|S\|\stackrel{\textup{\tiny{def}}}{=}
\max\limits_{\xi,\eta}\frac{\Big|\sum\limits_{1\leq p,q\leq n}s_{pq}\xi_q\eta_p\Big|}
{\sqrt{|\xi_1|^2+\,\cdots\,+|\xi_n|^2}\sqrt{|\eta_1|^2+\,\cdots\,+|\eta_n|^2}},
\end{equation}
where \(\max\) is taken over all complex numbers \(\xi_1,\,\ldots\,,\xi_n\) and \(\eta_1,\,\ldots\,,\eta_n\)
\begin{lem}
\label{maj1}
Let
\(S=(s_{pq})_1^n\in\mathfrak{M}_{n}\). Then the inequality
\begin{equation}
\label{In1}
\|S\|\leq\sum\limits_{1\leq{}p,q\leq n}|s_{pq}|
\end{equation}
holds.
\end{lem}
\begin{proof} The inequality \eqref{In1} is a direct consequence of the inequality
\begin{equation*}
\Big|\sum\limits_{1\leq p,q\leq n}s_{pq}\xi_q\eta_p\Big|\leq
\Big(\sum\limits_{1\leq{}p,q\leq n}|s_{pq}|\Big)\cdot
\max\limits_{1\leq{}p,q\leq{}n}\big|\xi_q\eta_p\big|
\end{equation*}
and of the inequality
\begin{equation*}
\max\limits_{1\leq{}p,q\leq{}n}\big|\xi_q\eta_p\big|\leq{}\sqrt{|\xi_1|^2+\,\cdots\,+|\xi_n|^2}
\sqrt{|\eta_1|^2+\,\cdots\,+|\eta_n|^2}. \qedhere
\end{equation*}
\end{proof}
\begin{lem}
\label{maj2}
Let
\(S=(s_{pq})_1^n\in\mathfrak{M}_{n}^{+}\). Then the inequality
\begin{equation}
\label{In2}
\sum\limits_{1\leq{}p,q\leq n}s_{pq}\leq{}n\cdot\|S\|
\end{equation}
holds.
\end{lem}
\begin{proof} The ratio \(\dfrac{\sum\limits_{1\leq{}p,q\leq n}s_{pq}}{n}\)
can be    considered as the ratio
\begin{equation*}
\frac{\sum\limits_{1\leq p,q\leq n}s_{pq}\xi_q\eta_p}
{\sqrt{|\xi_1|^2+\,\cdots\,+|\xi_n|^2}\sqrt{|\eta_1|^2+\,\cdots\,+|\eta_n|^2}}
\end{equation*}
with \(\xi_1=1,\,\ldots\,,\xi_n=1\) and \(\eta_1=1,\,\ldots\,,\eta_n=1\).
\end{proof}
The inequality expressed by following Lemma can be considered as \emph{an inverse
triangle inequality.} It holds for matrices with   \emph{non-negative} entries.
The total number of summands can be arbitrary large.
\begin{lem}
\label{ITI}
Let \(S_r\in\mathfrak{M}_n^{+},\,r=1,2,\,\ldots\,,m\). Then the inequality
\begin{equation}
\label{iti}
\sum\limits_{1\leq r\leq{}m}\|S_r\|\leq{}n\cdot\|\!\!\!\sum\limits_{1\leq r\leq{}m}S_r\|
\end{equation}
holds.
\end{lem}
\begin{proof}
Lemma \ref{ITI} is a direct consequence of Lemmas \ref{maj1} and \ref{maj2}.
\end{proof}
{\ }\\
The following technical result is used later in Section \ref{BTV}.

\begin{lem}
\label{TeL}
Let the following objects be given:
\begin{enumerate}
\item Matrices \(F_j\in\mathfrak{M}_n^{+},\, j=1,\,\ldots\,l,\) which satisfy the
condition
\begin{equation}
\label{Mcc}
\sum\limits_{1\leq j\leq l}F_j=I_n;
\end{equation}
\item
A matrix \(R\in\mathfrak{M}_n^{+}\) and a number \(N\in\mathbb{N}\).
\end{enumerate}
Then the inequality
\begin{equation}
\label{FMI}
\sum\limits_{k_1,k_2,\ldots,k_N}
\big\|F_{k_1}e^{R/N}F_{k_2}e^{R/N}\,\ldots\,F_{k_N}e^{R/N}\big\|\leq
n\cdot\big\|e^{R}\big\|
\end{equation}
holds. The summation in \eqref{FMI} is extended over all integers
\(k_1,k_2,\ldots,k_N\) which satisfy the conditions\footnote{
So the sum in \eqref{FMI} contains \(l^N\) summands.
}
\(1\leq{}k_1\leq{}l,\,1\leq{}k_2\leq~l, \, \ldots\,,1\leq{}k_N\leq{}l.\)
\end{lem}
\begin{proof}
According to Lemma \ref{ITI}, the inequality
\begin{multline*}
\sum\limits_{k_1,k_2,\ldots,k_N}
\big\|F_{k_1}e^{R/N}F_{k_2}e^{R/N}\,\ldots\,F_{k_N}e^{R/N}\big\|\leq\\
n\cdot\Big\|\sum\limits_{k_1,k_2,\ldots,k_N}
F_{k_1}e^{R/N}F_{k_2}e^{R/N}\,\ldots\,F_{k_N}e^{R/N}\,\Big\|
\end{multline*}
holds. From the condition \eqref{Mcc} it follow that
\begin{equation*}
\sum\limits_{k_1,k_2,\ldots,k_N}
F_{k_1}e^{R/N}F_{k_2}e^{R/N}\,\ldots\,F_{k_N}e^{R/N}=
e^{R/N}\cdot{}e^{R/N}\cdot\,\cdots\,\cdot{}e^{R/N}=e^{R}.
\end{equation*}
\end{proof}
\begin{rem} If \(F_j\in\mathfrak{M_n^{+}},\,\forall j=1,\,\ldots\,l,\) and the equality \eqref{Mcc} holds, then \(F_j\in\mathfrak{D_n},\,\forall\, j=1,\,\ldots\,l\).
The diagonal entries of each of the matrices \(F_j\) belong to the interval
\([0,1]\).
\end{rem}
\section{An expression for the total variation of the measure \(\boldsymbol{M_N(d\lambda)}\).}
Since the measure \(M_N(d\lambda)\) is discrete and its support is a finite set
\(ch_N(\sigma(A))\), the total variation of this measure is expressed by the sum
\[\sum\limits_{\lambda\in{}ch_N(\sigma(A))}\|{}M_N(\{\lambda\})\|.
\]
To prove that the family of measures \(\big\{M_{N}(d\lambda)\big\}_{1\leq{}N<\infty}\)
is weakly convergent, we have to obtain an estimate of the form
\begin{equation}
\label{EV}
\sum\limits_{\lambda\in{}ch_N(\sigma(A))}\|{}M_N(\{\lambda\})\|\leq C,
\quad \forall\, N,
\end{equation}
where \(C<\infty\) does not depends on \(N\).
\begin{lem}
\label{reg}
Let \(M_N(\{\lambda\})\), \(\lambda\in{}ch_N(\sigma(A))\), be the matrices which appear in the representation
\eqref{EEa}-\eqref{EEb} of the \(N\)-approximant \(L_{N}(t)\). Then the inequality
\begin{equation}
\label{IE}
\sum\limits_{\lambda\in{}ch_N(\sigma(A))}\|{}M_N(\{\lambda\})\|\leq{}
\sum\limits_{k_1,\ldots,k_N}\| M_{k_1,\ldots,k_N} \|{},
\end{equation}
holds, where \(M_{k_1,\ldots,k_N}\) are the same as in \eqref{Pr}. On the right hand side of \eqref{IE}, the summation is extended over all integers \(k_1,\ldots,k_N\) satisfying
the conditions \(1\leq{}k_1\leq{}l,\,\ldots\,,1\leq{}k_N\leq{}l\).
\end{lem}
\begin{proof}
Applying the triangle inequality to \eqref{EEb}, we obtain the inequality
\begin{equation}
\label{IEe}
\|{}M_N(\{\lambda\})\|\leq{}
\sum\limits_{k_1,\ldots,k_N}\| M_{k_1,\ldots,k_N} \|,\qquad \forall\,\lambda\in{}Nch(\sigma(A)).
\end{equation}
 On the right hand side of \eqref{IEe}, the summation is extended over all those integers       \(k_1,\,\ldots\,,k_N\) for which
\(\tfrac{\lambda_{k_1}+\,\ldots\,+\lambda_{k_N}}{N}=\lambda\). Adding the inequalities \eqref{IEe} over all \(\lambda\in{}ch_N(\sigma(A))\), we come to the inequality
\begin{equation}
\label{IEf}
\sum\limits_{\lambda\in{}ch_N(\sigma(A))}\|{}M_N(\{\lambda\})\|\leq{}
\sum\limits_{\lambda\in{}ch_N(\sigma(A))}\Big(\sum\limits_{k_1,\ldots,k_N}\| M_{k_1,\ldots,k_N} \|\Big).
\end{equation}
Regrouping summands in the right hand side of \eqref{IEf}, we come to the inequality
\eqref{IE}.
\end{proof}

\section{The subordination relation.}
\begin{defn}
Let \(M=(m_{pq})_1^n\in\mathfrak{M}_n\), \(S=(s_{pq})_1^n\in\mathfrak{M}_n^{+}\). We say that \emph{the matrix \(M\) is subordinated to the matrix \(S\)} and use \emph{the notation
\(M\preceq{}S\) for the subordination relation} if  the inequalities
\begin{equation}
\label{So}
|m_{pq}|\leq{}s_{pq},\quad 1\leq{}p,q\leq n,
\end{equation}
hold for the entries \(m_{pq}\), \(s_{pq}\) of the matrices \(M,S\), respectively.
\end{defn}
\begin{lem}
\label{NEs}
We assume that \(M\in\mathfrak{M}_n, S\in\mathfrak{M}_n^{+}\), and \(M\preceq{}S\).
Then
\begin{equation}
\label{InE}
\|M\|\leq{}\|S\|.
\end{equation}
\end{lem}
\begin{proof} Let \(m_{pq},s_{pq}\) be the entries of the matrices \(M\) and \(S\), and \(\xi_1,\,\ldots\,,\xi_n\), \(\eta_1,\,\ldots\,,\eta_n\) be arbitrary complex numbers. Then the inequality
\begin{multline*}
\big|\sum\limits_{1\leq p,q\leq n}m_{pq}\xi_p\eta_q\big|\leq
\sum\limits_{1\leq p,q\leq n}|m_{pq}||\xi_p||\eta_q|\leq
\sum\limits_{1\leq p,q\leq n}s_{pq}|\xi_p||\eta_q|\\
\leq{}\|S\|\sqrt{|\xi_1|^2+\,\cdots\,+|\xi_n|^2}
\sqrt{|\eta_1|^2+\,\cdots\,+|\eta_n|^2}.
\end{multline*}
holds.
\end{proof}
\begin{defn}
\label{USN}
Given a matrix \(B\in\mathfrak{M}_n\), we associate the matrix \(R(B)\)
with \(B\). By definition,
\begin{equation}
\label{SoR}
R(B)=(r_{pq})_{1}^{n},\qquad  r_{pq}=\|B\|,\quad 1\leq{}p,q\leq n,
\end{equation}
\end{defn}
\begin{lem}
\label{Sur}
The matrix \(B\) is subordinated to the
matrix \(R(B)\).
\end{lem}
\begin{proof} The entry \(b_{pq}\) of the matrix \(B\) satisfies the inequality \(|b_{pq}|\leq\|B\|=r_{pq},\,1\leq p,q\leq{}n\).
\end{proof}
\begin{lem}{\ }
\label{Au}
\begin{enumerate}
\item
The matrix \(R(B)\) is a Hermitian matrix of rank one.
\item The norms of the matrix \(R(B)\) and its exponential \(e^{R(B)}\) are
\begin{equation}
\label{NeB}
\|R(B)\|=n\|B\|,\quad \|e^{R(B)}\|=e^{n\|B\|}.
\end{equation}
\end{enumerate}
\end{lem}
\begin{proof}
The only non-zero eigenvalue of the matrix \(R(B)\) is the number \(n\|B\|\).
\end{proof}

\begin{lem}
\label{sol}
Let \(\Psi_k\in\mathfrak{M}_n,\,\Phi_k\in\mathfrak{M}_n^{+},\,k=1,\,\ldots\,,m\).
Assume that for each \(k=1,\,\ldots\,,m\), the matrix \(\Psi_k\) is subordinated to the
matrix \(\Phi_k\):
\[\Psi_{k}\preceq\Phi_k,\quad k=1,\,\ldots\,,m.\]
Then the subordination relations
\begin{align*}
\Psi_1+\Psi_2+\,\cdots\,+\Psi_m &\preceq\Phi_1+\Phi_2+\,\cdots\,+\Phi_m,\\
\Psi_1\cdot\Psi_2\cdot\,\cdots\,\cdot\Psi_m &\preceq\Phi_1\cdot\Phi_2\cdot\,\cdots\,\cdot\Phi_m
\end{align*}
hold for the sum and the product of these matrices.
\begin{proof} The assertion of Lemma is a direct consequence of the definition of matrix addition and multiplication and of elementary properties of numerical inequalities.
\end{proof}
\end{lem}
\begin{lem}
\label{SuE}
 Let \(X\in\mathfrak{M}_n^{+}\). Then
\(e^{X}\in\mathfrak{M}_n^{+}\).
If \(Y\in\mathfrak{M}_n\),
\(Y\preceq{}X\), then \(e^{Y}\preceq{}e^{X}\).
\end{lem}
\begin{proof}   According Lemma \ref{sol}, the subordination relations
\(\frac{1}{m!}Y^{m}\preceq{}\frac{1}{m!}X^{m}\) hold for every \(m=0,1,2,\,\ldots\).
Using Lemma \ref{sol} once more, we conclude that
\(\sum\limits_{0\leq m<\infty}\frac{1}{m!}Y^{m}\preceq\sum\limits_{0\leq m<\infty}\frac{1}{m!}X^{m}\).
\end{proof}

\section{A bound for the total variation of the measure \(\boldsymbol{M_N(d\lambda)}\).}
\label{BTV}
\begin{lem}
\label{CruE}
Let \(A\in\mathfrak{H}_n\), \(B\in\mathfrak{M}_n\), \(N\in\mathbb{N}\), and let the matrices
\(M_{k_1,\,\ldots\,,k_N}\) be defined according to \eqref{Pr}.

Then the inequality
\begin{equation}
\label{crue}
\sum\limits_{1\leq{}k_1,\ldots,k_N\leq{}l}\| M_{k_1,\ldots,k_N} \|\leq{}ne^{n\|B\|}
\end{equation}
holds.
\end{lem}
\begin{proof} {\ }\\
\textbf{1.} We impose the additional condition: the matrix \(A\) is diagonal. So
\begin{equation}
\label{AC}
A\in\mathfrak{H}_n\cap\mathfrak{D}_n.
\end{equation}
\emph{Then all spectral projectors \(E_{\lambda_j}\) are diagonal matrices}. Hence
\(E_{\lambda_j}\in\mathfrak{M}_n^{+}\), i.e.
\begin{equation}
\label{SpP}
E_{\lambda_j}\preceq{}E_{\lambda_j},\quad j=1,\,\ldots\,,l.
\end{equation}
Let the matrix \(R(B)\) be defined according to Definition \ref{USN}.
By Lemma \ref{Sur}, \(B\preceq{}R(B)\). By Lemma \ref{SuE},
\begin{equation}
\label{Sue}
e^{B/N}\preceq{}e^{R(B)/N}.
\end{equation}
By Lemma \ref{sol}, the subordination relation
\begin{equation*}
M_{k_1\ldots{}k_N}\preceq{}
E_{\lambda_{k_1}}e^{R(B)/N}E_{\lambda_{k_2}}e^{R(B)/N}\,\ldots\,
E_{\lambda_{k_N}}e^{R(B)/N}
\end{equation*}
is satisfied for every \(k_1,\,\ldots\,,k_N\). By Lemma \ref{NEs}, the inequality
\begin{equation*}
\|M_{k_1\ldots{}k_N}\|\leq\|E_{\lambda_{k_1}}e^{R(B)/N}E_{\lambda_{k_2}}e^{R(B)/N}\,\ldots\,
E_{\lambda_{k_N}}e^{R(B)/N} \|
\end{equation*}
holds. Adding the above inequalities, we obtain the inequalty
\begin{multline}
\label{rhs}
\sum\limits_{k_1,\ldots,k_N}\| M_{k_1,\ldots,k_N} \|\leq\\
\sum\limits_{k_1,\ldots,k_N}\big\|E_{\lambda_{k_1}}e^{R(B)/N}E_{\lambda_{k_2}}e^{R(B)/N}\,\ldots\,
E_{\lambda_{k_N}}e^{R(B)/N}\big \|
\end{multline}
To estimate the sum in the right hand side of \eqref{rhs}, we apply Lemma \ref{TeL}
with \(F_j=E_{\lambda_j}, R=R(B)\) and obtain the inequality
\begin{equation}
\label{subs}
\sum\limits_{k_1,\ldots,k_N}\big\|E_{\lambda_{k_1}}e^{R(B)/N}E_{\lambda_{k_2}}e^{R(B)/N}\,\ldots\,
E_{\lambda_{k_N}}e^{R(B)/N}\big \|\leq{}n\|e^{R(B)}\|.
\end{equation}
Now we refer to Lemma \ref{Au}. The inequality \eqref{crue} is a consequence of
\eqref{rhs}, \eqref{subs} and \eqref{NeB}.\\
\textbf{2}. The inequality \eqref{crue} is proved under the extra assumption that the matrix \(A\) is diagonal. Now we get rid of the extra assumption, that the matrix \(A\)
is diagonal.

Let \(A\) be an arbitrary matrix from \(\mathfrak{H}_n\). There exists a unitary matrix
\(U\) such that the matrix
\begin{equation}
\label{diA}
A^d=UAU^{\ast}
\end{equation}
is diagonal. Of course \(A^d\in\mathfrak{H}_n\).
Then we define the matrices
\begin{equation}
\label{diM}
B^d=UBU^{\ast}, \quad E_{\lambda_{j}}^d=UE_{\lambda_{j}}U^\ast,\quad
 M_{k_1,\ldots,k_N}^d=U M_{k_1,\ldots,k_N}U^{\ast}.
\end{equation}
The matrices \(E_{\lambda_{j}}^d\) are the spectral projectors of the matrix \(A^d\).
The matrices \(M_{k_1,\ldots,k_N}^d\) can be represented in the form
\begin{equation}
\label{MdR}
M_{k_1,\,\ldots,\,k_N}^d=E_{\lambda_{k_1}}^de^{B^d/N}E_{\lambda_{k_2}}^de^{B^d/N}\,\ldots\,
E_{\lambda_{k_N}}^de^{B^d/N}
\end{equation}
Since the matrix \(A^d\) is diagonal, the inequality
\begin{equation}
\label{crued}
\sum\limits_{k_1,\ldots,k_N}\| M_{k_1,\ldots,k_N}^d \|\leq{}ne^{n\|B^d\|}
\end{equation}
holds. It remains to note that
\begin{equation*}
\|M_{k_1,\,\ldots,\,k_N}^d\|=\|M_{k_1,\,\ldots,\,k_N}\|,\quad \|B^d\|=\|B\|.\qedhere
\end{equation*}
\end{proof}
\begin{lem}
\label{BVM}
Given the matrices \(A\in\mathfrak{H}_n,\,B\in\mathfrak{M}_n\), let \(M_N(d\lambda)\)
be the matrix valued measure which appears in the representation \eqref{LRN} of the
\(N\)-approximant \(L_N(t)\) of the matrix function \(e^{tA+B}\).

The total variation of the measure  \(M_N(d\lambda)\) admits the bound
\begin{equation}\label{AIn}
\sum\limits_{\lambda\in{}ch_N(\sigma(A))}\|{}M_N(\{\lambda\})\|\leq{}ne^{n\|B\|}.
\end{equation}
\end{lem}
\begin{proof} We combine the inequalities \eqref{IE} and \eqref{crue}.
\end{proof}
\section{The representation of the matrix function \(\boldsymbol{e^{tA+B}}\)
in the form of the Laplace transform of a matrix valued measure.}
\begin{thm}
\label{MaT}
Let  matrices \(A\) and \(B\) be given, \(A\in\mathfrak{H}_n\), \(B\in\mathfrak{M}_n\). Let \(\lambda_{\textup{min}}\) and
    \(\lambda_{\textup{max}}\) be the smallest and the largest eigenvalues of
    the matrix \(A\), \(\mathfrak{B}\)  be the class of Borel sets of the closed
    interval \([\lambda_{\textup{min}},\lambda_{\textup{min}}]\).

    Then there exists a set function \(M:\mathfrak{B}\to\mathfrak{M}_n\) such that:
    \begin{enumerate}
    \item \(M\) is countably additive and regular on \(\mathfrak{B}\), i.e. \(M\) is a regular Borel \(\mathfrak{M}_n\)-valued measure on \([\lambda_{\textup{min}},\lambda_{\textup{min}}]\).
    \item The equality
    \begin{equation}
    \label{MaEq}
    e^{tA+B}=\int\limits_{[\lambda_{\textup{min}},\lambda_{\textup{min}}]}
    e^{t\lambda}M(d\lambda), \quad \forall\,t\in\mathbb{C},
    \end{equation}
    holds.
    \item If \(B\in\mathfrak{H}_n\), then the measure \(M\) is \(\mathfrak{H}_n\)-valued, i.e. \(M:\mathfrak{B}\to\mathfrak{H}_n\).
    \end{enumerate}
\end{thm}
\begin{proof}
Let us consider the Banach space \(C([\lambda_{\textup{min}},\lambda_{\textup{max}}])\)
of \(\mathbb{C}\)-valued continuous functions on
the interval \([\lambda_{\textup{min}},\lambda_{\textup{max}}]\)
equipped with the standard norm
\begin{equation*}
\|x(\lambda)\|=\max\limits_{\lambda\in[\lambda_{\textup{min}},\lambda_{\textup{max}}]}
|x(\lambda)|, \quad x(\lambda)\in{}C([\lambda_{\textup{min}},\lambda_{\textup{max}}]).
\end{equation*}
The support of each of the measures \(M_N(d\lambda)\) is contained in
the closed interval  \([\lambda_{\textup{min}},\lambda_{\textup{max}}]\).
(See Remark \ref{con}.) The total variation of the measures
\(M_N(d\lambda)\) is bounded from above by some finite value which does not depend
on \(N\). (See Lemma \ref{BVM}.) According to \eqref{LRI}, for each \(t\in\mathbb{R}\)
there exists the limit
\begin{equation*}
\lim\limits_{N\to\infty}
\int\limits_{[\lambda_{\textup{min}},\lambda_{\textup{max}}]}e^{t\lambda}M_N(d\lambda).
\end{equation*}
The system of  functions \(\{e^{t\lambda}\}_{t\in\mathbb{R}}\) is complete in the space \(C([\lambda_{\textup{min}},\lambda_{\textup{max}}])\). Therefore for each
\(x(\lambda)\in{}C([\lambda_{\textup{min}},\lambda_{\textup{max}}])\)
 the limit
\begin{equation}
\label{FE}
J(x)=\lim\limits_{N\to\infty}
\int\limits_{[\lambda_{\textup{min}},\lambda_{\textup{max}}]}x(\lambda)\,M_N(d\lambda)
\end{equation}
exists.
The mapping \(J:C([\lambda_{\textup{min}},\lambda_{\textup{max}}])\to\mathfrak{M}_n\)
is a continuous linear mapping.

Let \(M(d\lambda)\) be the weak limit of the sequence of measures \(M_N(d\lambda)\).
The \(\mathfrak{M}_n\)-valued measure \(M(d\lambda)\) gives the integral representation of the mapping~\(J\):
\begin{equation}
\label{FEr}
J(x)=
\int\limits_{[\lambda_{\textup{min}},\lambda_{\textup{max}}]}x(\lambda)\,M(d\lambda),
\quad x(\lambda)\in{}C([\lambda_{\textup{min}},\lambda_{\textup{max}}]).
\end{equation}
In view of \eqref{LRI},
\[J(e^{t\lambda})=e^{tA+B}.\]
Thus the representation \eqref{MaEq} is established.

If the matrix \(B\) is Hermitian: \(B\in\mathfrak{H}_n\), then
\begin{equation*}
e^{tA+B}=\big(e^{tA+B}\big)^{\ast},\quad \forall\,t\in\mathbb{R}.
\end{equation*}
Hence
\[\int\limits_{[\lambda_{\textup{min}},\lambda_{\textup{min}}]}
    e^{t\lambda}M(d\lambda)=\int\limits_{[\lambda_{\textup{min}},\lambda_{\textup{min}}]}
    e^{t\lambda}(M(d\lambda))^{\ast},\quad    \forall\,t\in\mathbb{R}.\]
Since the system \(\{e^{t\lambda}\}_{t\in\mathbb{R}}\) is complete in the space \(C([\lambda_{\textup{min}},\lambda_{\textup{max}}])\), the measures \(M(d\lambda)\)
and  \((M(d\lambda))^{\ast}\) must coincide. In other words, the measure   \(M(d\lambda)\) is \(\mathfrak{H}_n\)-valued.
\end{proof}

\section{The measure \(\boldsymbol{M(d\lambda)}\) is not necessarily non-negative.}
\label{NNN}
\begin{defn}\label{nNeg}
Let
\(S=(s_{pq})_1^n\in\mathfrak{M}_{n}\). The matrix \(S\) is said to be \emph{non-negative}
if   the inequality
\begin{equation*}
\sum\limits_{1\leq p,q\leq n}s_{pq}\xi_p\overline{\xi_q}\geq 0
\end{equation*}
holds for all complex numbers \(\xi_1,\,\ldots\,,\xi_n\).
\end{defn}
\begin{defn}\label{nNeMe}
Let \(\mathfrak{B}\) be the class of Borel sets of \(\mathbb{R}\), \(M(d\lambda):\mathfrak{B}\to\mathfrak{M}_n\)
be a matrix valued measure. The measure \(M(d\lambda)\) is said to be \emph{non-negative}
if the matrix \(M(\delta)\) is non-negative for every set \(\delta\in\mathfrak{B}\).
\end{defn}

Comparing the equalities \eqref{StR} and \eqref{MaEq}, we conclude that
\begin{equation}
\mu(d\lambda)=\tr{}M(d\lambda).
\end{equation}

The following question arises naturally.\\[2.0ex]
 \textbf{Question}. \emph{Let \(A\in\mathfrak{H}_n\), \(B\in\mathfrak{H}_n\), and \(M(d\lambda)\) be the \(\mathfrak{H}_n\)-valued
measure which appears in the representation \eqref{MaEq} of the function
\(e^{tA+B}\).
Is the measure \(M(d\lambda)\) non-negative}?

The following example shows that the answer to this question is negative already for
\(n=2\).\\[2.0ex]
\textbf{Example}.
Let
\begin{equation}
A=
\begin{bmatrix}
2&0\\
0&0
\end{bmatrix},
\qquad
B=
\begin{bmatrix}
0&1\\
1&0
\end{bmatrix}
\cdot
\end{equation}
The eigenvalues of the matrix \(tA+B\) are
\begin{equation}
\lambda_1(t)=t+\sqrt{t^2+1},\qquad \lambda_2(t)=t-\sqrt{t^2+1}.
\end{equation}
The spectral projectors of the matrix \(tA+B\) corresponding to these eigenvalues are
\begin{equation}
E_1(t)=
\begin{bmatrix}
\frac{\sqrt{t^2+1}+t}{2\sqrt{t^2+1}}&\frac{1}{2\sqrt{t^2+1}}\\[1.5ex]
\frac{1}{2\sqrt{t^2+1}}&\frac{\sqrt{t^2+1}-t}{2\sqrt{t^2+1}}
\end{bmatrix}
\ccomma \quad
E_2(t)=
\begin{bmatrix}
\frac{\sqrt{t^2+1}-t}{2\sqrt{t^2+1}}&-\frac{1}{2\sqrt{t^2+1}}\\[1.5ex]
-\frac{1}{2\sqrt{t^2+1}}&\frac{\sqrt{t^2+1}+t}{2\sqrt{t^2+1}}
\end{bmatrix}
\cdot
\end{equation}
The matrix \(e^{tA+B}\) can be calculated explicitly:
\begin{equation}
e^{tA+B}=e^{\lambda_1(t)}E_1(t)+e^{\lambda_2(t)}E_2(t).
\end{equation}
In particular,
\begin{equation}\label{EEM}
\Big(\frac{d{\ }}{dt}e^{tA+B}\Big)_{|_{t=0}}=D,
\end{equation}
where
\begin{equation}
D=
\begin{bmatrix}
e& \dfrac{e-e^{-1}}{2}\\[2.0ex]
\dfrac{e-e^{-1}}{2}&e^{-1}
\end{bmatrix}
\cdot
\end{equation}
The matrix \(D\) is not non-negative:
\begin{equation}
\textup{det}\,D=\frac{6-e^2-e^{-2}}{4}<0.
\end{equation}
From \eqref{MaEq} it follows that
\begin{equation}
D=\int\limits_{[0,2]}\lambda{}M(d\lambda).
\end{equation}
If the measure \(M(d\lambda)\) would be non-negative, then the matrix \(D\) would be non-negative.

\end{document}